\documentclass[%
 aip,
 sd,%
 amsmath,amssymb,
 reprint,%
]{revtex4-1}

\usepackage{graphicx}
\usepackage{dcolumn}
\usepackage{bm}
\usepackage{amsthm}

\newtheorem{theorem}{Theorem}[section]
\newtheorem{lemma}[theorem]{Lemma}
\newtheorem{proposition}[theorem]{Proposition}

\newtheorem{remark}[theorem]{Remark}
\DeclareMathOperator{\cosec}{cosec}

\begin{document}

\preprint{APS/123-QED}

\title{Linear and nonlinear stability of periodic orbits in annular billiards}

\author{Carl P. Dettmann}
 \email{Carl.Dettmann@bris.ac.uk}
\author{Vitaly Fain}%
 \email{vf13950@bristol.ac.uk}
\affiliation{%
 University of Bristol, School of Mathematics,
 University Walk,\\ Bristol BS8 1TW, UK 
}%




\date{\today}

\begin{abstract}
An annular billiard is a dynamical system in which a particle moves freely in a disk except for elastic collisions with the boundary, and also a circular scatterer in the interior of the disk. We investigate stability properties of some periodic orbits in annular billiards in which the scatterer is touching or close to the boundary. We analytically show that there exist linearly stable periodic orbits of arbitrary period for scatterers with decreasing radii that are located near the boundary of the disk. As the position of the scatterer moves away from a symmetry line of a periodic orbit, the stability of periodic orbits changes from elliptic to hyperbolic, corresponding to a saddle-center bifurcation. When the scatterer is tangent to the boundary, the periodic orbit is parabolic. We prove that slightly changing the reflection angle of the orbit in the tangential situation leads to the existence of KAM islands. Thus we show that there exists a decreasing to zero sequence of open intervals of scatterer radii, along which the billiard table is not ergodic.
\end{abstract}

\maketitle


\begin{quotation}A billiard is a dynamical system where a point particle moves with constant velocity inside a domain and experiences elastic collisions with the boundary of the domain. The shape of the boundary determines the dynamics of the billiard. Billiards in a disk on a plane are completely integrable, while annular billiard tables consisting of a particle confined between two nonconcentric disks generically display mixed phase space due to a family of regular orbits that never touch the scatterer. Billiard models find applications in a variety of problems in statistical \cite{dettmann2014diffusion}, classical and quantum \cite{haake2013quantum} physics. In this paper, we consider annular billiard tables formed of a small circular scatterer placed in the interior of a unit circle; this is a popular geometry for microwave billiard experiments \cite{bittner2014scattering}. Circular boundaries allow us to analytically examine linear and nonlinear stability of some periodic orbits. Depending on the parameters of the problem, we find that there exist linearly stable orbits of arbitrarily large period. We show the existence of a saddle-center bifurcation as the parameters vary, corresponding to a change of stability from linearly elliptic to saddle type. Placing the scatterer tangentially to the external circle creates a cusp that is a source of singularities in the billiard. We use KAM theory to establish that in the cusp case, the periodic orbits are nonlinearly stable.
\end{quotation}

\section{\label{sec:level1}Introduction and main results}

Billiards are dynamical systems modelling the motion of a classical particle moving with constant speed inside a bounded domain and performing elastic collisions with the boundary of the domain. They display a whole spectrum of dynamical behaviour ranging from completely integrable to chaotic. The mathematical study of billiards was initiated by Birkhoff \cite{Birkhoff:1927p17467}, and later significantly extended by Sinai \cite{sinai1970dynamical} and his followers. Billiards arise in models for various physical phenomena, for example in statistical mechanics models of hard balls due to Boltzmann \cite{sinai1979development}. A  billiard in a plane consists of a classical point particle moving with constant velocity in a bounded domain $Q \subset \mathbb{R}^{2}$, called the billiard table, and obeying the optical reflection law upon collisions with the boundary of the billiard table $\partial Q$. The shape of the boundary determines the dynamics of the billiard.

It was proved by Birkhoff \cite{Birkhoff:1927p17467} that elliptic billiard tables are integrable. A long standing Birkhoff's conjecture, in fact, states that elliptic billiards are the only types of completely integrable strictly convex tables. Recently it was shown by Avila, Kaloshin and De Simoi \cite{avila2014integrable} that this conjecture is true for small perturbations of elliptic billiards with small eccentricity.  By using KAM theory, Lazutkin \cite{lazutkin1973existence} proved that existence of a continuum set of caustics near the boundary of strictly convex $C^{553}$ boundaries, thus preventing ergodicity. Douady \cite{douady1982application} refined this result to $C^{6}$ boundaries.  On the other hand, it was shown by Sinai \cite{sinai1970dynamical} that concave billiard tables are ergodic and hyperbolic, while later Bunimovich \cite{bunimovich1974ergodic}, by using the defocusing mechanism, showed that certain piecewise smooth convex table (i.e. the stadium) are also hyperbolic and ergodic. It has been also recently conjectured by Bunimovich and Grigo \cite{grigobilliards, bunimovich2010focusing} that the presence of absolute focusing components is a requirement for ergodicity.

As noted in Foltin \cite{foltin2002billiards} the method of defocusing requires the circular arcs of the boundary to be disjoint, and thus does not apply to strictly convex billiard tables with inner scatterers.  It was shown by Foltin \cite{foltin2002billiards} and independently by Chen \cite{chen2010topological}  that the generically, strictly convex $C^{2}$ tables with small inner scatterers admit positive topological entropy.

In the class of convex billiards with scatterers, perhaps the simplest geometry is that of annular billiard, that is, a circle billiard with a smaller inner scatterer. There appears to be a lack of published mathematically rigorous studies of billiards of this type.  Analytical and numerical methods were used to catalogue some symmetric periodic orbits up to order $6$ in annular billiards in the work by Gouesbet et al \cite{gouesbet2001periodic}, while coexistence of KAM islands and chaotic motions in annular billiards were studied numerically in Saito et al\cite{saito1982numerical}. Recently, the related work of Correia and Zhang \cite{correia2015stability} demonstrated the existence of stability of some periodic orbits in so-called \textit{moon} billiards, and ergodicity of certain other tables in that class. Linear stability and bifurcations of some periodic orbits in oval and elliptic billiards with an inner scatterer were investigated by  da Costa et al  \cite{da2015dynamics}. Marginally unstable periodic orbits and relation to quantum chaos has been investigated by Altmann et al \cite{altmann2008prevalence}.

In this work we show that there exist certain linearly stable periodic orbits of arbitrarily large period in a circle billiard with a small interior scatterer. Furthermore we prove that in the case of the scatterer being tangential to the outer circle, the periodic orbits can be made to be nonlinearly (KAM) stable.  

Take a unit disk $D$ in the plane with boundary $\partial D$. The billiard in $D$ is completely integrable \cite{chernov2006chaotic}. For every positive integer $n \geq 3$, the billiard trajectory with the angle of reflection $\theta = \frac{\pi}{n}$ made with the positively oriented tangent to $\partial D$ is $n$-periodic, tracing an $n$-sided regular polygon inscribed in $D$. The billiard trajectory with reflection angle $\theta = \frac{k\pi}{n}$ where $k$ is an integer,  $1 \leq k \leq \frac{n}{2}$, is also $n$-periodic but traces a $n$-pointed star polygon inscribed in $D$ if $(k,n)$ are coprime. Let us fix $n$ and $k$. We obtain the annular billiard table by placing an inner circular scatterer $D_{R}$, of small radius $R \ll 1$ in the interior of $D$, centered on the middle of one of the sides of the polygon. Thus $D_{R}$ is normal to the billiard path.  Let the boundary of $D_{R}$ be $\partial D_{R}$. Since the circle and the $n$-polygon is rotationally symmetric, it makes no difference on which side of the polygon $D_{R}$ is located. It is possible to perturb this configuration in two ways. One may vary $R$ up to some maximum \textit{admissible} value (to be specified in Section~\ref{sec:level3}) to ensure that $D_{R}$ is in interior of $D$, and in the case of star polygonal orbits, to avoid other sides of the same polygon.  Another perturbation would be to make small displacements $\delta \geq 0$ of $D_{R}$ along the side of the polygon away from the centre of the initial position of $D_{R}$, as long as $D_{R}$ stays in the interior of $D$. Therefore, the maximum value of $R$ depends on $n$, $k$ and $\delta$, and we suppress this dependence for clarity of presentation. We will call the corresponding annular billiard table $Q(R)$.

With the scatterer located as described above, we obtain a $(2n+2)$-period orbit, we call it a type (a) orbit (see Fig.~\ref{fig1}.), in the following way. The billiard will undergo $n$ consecutive collisions with $\partial D$, with the initial angle of reflection made with $\partial D$ chosen to be $\theta_{i} =  \frac{k \pi}{n}$ for $i \in \lbrace 0,...,n-1 \rbrace $. Suppose $D_{R}$ is located on the straight line billiard trajectory segment joining the $n$-th and $(n+1)$-th collision points. Then $(n+1)$-th collision is perpendicularly on $\partial D_{R}$. After collision with $\partial D_{R}$ the particle reverses its path, and the $(n+2)$-th collision is again with $\partial D$. The particle now performs $(n-1)$ collisions on $\partial D$ again, before colliding with $\partial D_{R}$ perpendicularly, and bouncing back to form a closed orbit of length $(2n+2)$. For the reversed direction of the trajectory we have $\theta_{i} = \pi - \frac{k\pi}{n}$ for  $ i \in \lbrace n+1,..., 2n \rbrace$. Let $\gamma_{a,k}$ denote the type (a) orbit corresponding to  fixed $k \geq 1 $ for a given $n$. We suppress the dependence of the orbit $\gamma_{a,k}$ on the parameters $n$, $R$ and $\delta$.
 
For case $k=1$ and $\delta=0$, one may take a certain maximum $R$ such that $\partial D_{R}$ is \textit{tangent} to $\partial D$ (thus forming a \textit{cusp}). It is known that cusps can be a source of singularities in billiards \cite{chernov2007dispersing}. At the present time, cusps created by one focusing and one dispering boundaries have not received much attention in the literature except in the recent work \cite{da2015dynamics}. Prior to that publication, studies were limited to the situation with two dispersing or one dispersing and one flat wall \cite{balint2011limit}, \cite{chernov2007dispersing}, \cite{balint2008decay}. In the cusp case, $R$ depends on $n$ only and we obtain a one-parameter family of $(2n+2)$-periodic orbits for a annular cusp billiard.

\begin{figure} [b]
\centering
\includegraphics[width=8 cm]{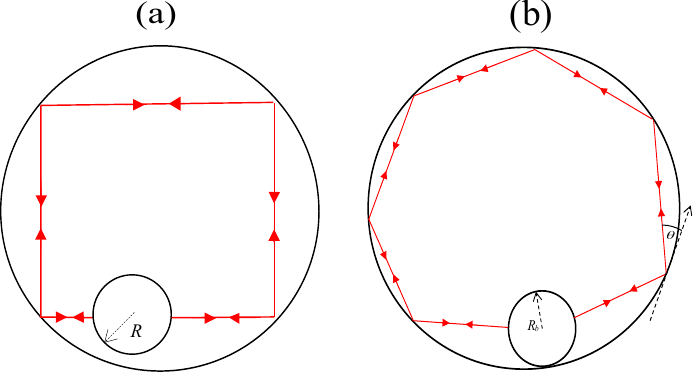}
\vspace*{-1mm}
\caption[]{Type (a) orbit with $\theta = \frac{k\pi}{n}$ with $k=1$, $n=4$, $\delta>0$, and a type (b) orbit with $\theta = \frac{\pi}{n} + \epsilon$; $R_{b}$ as in eq. \protect (\ref{radiusb}) below}
\label{fig1}
\end{figure}

We have the following theorem concerning the linear stability of periodic orbits for type (a).
 
\begin{theorem}
For any given $n \geq 3$ there exists a billiard table $Q(R)$ such that the orbit $\gamma_{a,k}$ is linearly stable for certain choices of $R$, $k<7$ and small $\delta \neq 0$.  Specifically, $\gamma_{a,1}$ is linearly stable when $\delta \neq 0$ for $R$ as in Proposition~\ref{prop1}. When $(k,n)$ are coprime, $\gamma_{a,k}$ is linearly stable for $\delta \neq 0$ and $n \geq n_{k}$ with $n_{k}$ and $R$ as in Proposition~\ref{prop2}. When  $\delta = 0$, $\gamma_{a,k}$ is neutrally stable for all $n$ and $k$ at any admissible $R$, and also when $\delta \neq 0$ for $R =\frac{\sin \frac{k\pi}{n} + \sqrt{\sin^{2} \frac{k\pi}{n} + 4n^{2}\delta^{2}}}{2n}$.
\label{theorem1}
\end{theorem}

The proof of the Theorem~\ref{theorem1} is in Section~\ref{sec:level3} of the paper. Propositions~\ref{prop1} and ~\ref{prop2} make up Theorem~\ref{theorem1}.

We also introduce another type (b) of $(2n+2)$-periodic orbits, with $n \geq 3$, (also see Fig.~\ref{fig1}) by slightly changing the initial reflection angle of the billiard trajectory away from $\frac{\pi}{n}$ by some small $\epsilon>0$ such that $\theta = \frac{\pi}{n} + \epsilon$ is not $\pi$-rational. In this case, the orbit in the $D$ (without $D_{R}$) is not periodic, and the polygon traced by the billiard path does not quite close. We position $\partial D_{R}$ tangentially to $\partial D$ and perpendicularly to the billiard path as before, creating a closed $(2n+2)$ orbit. Thus, type (b) orbits may be created from $\gamma_{a,1}$ by changing the angle of reflection $\frac{\pi}{n}$ slightly. Since the scatterer is tangent to $D$, the radius is of the scatterer is defined by the choice of $n$ and $\epsilon$, and will be specified in Section~\ref{sec:level4}. Denote the radius by $R_{b}$ and the scatterer by $D_{R_{b}}$ for this situation. For every $n \geq 3$, the corresponding billiard table is denoted $Q(R_{b})_{n, \epsilon}$. Periodic orbit corresponding to type (b) will be denoted $\gamma_{b}$. We suppress the dependence of $\gamma_{b}$ on $n$. 
 
For the type (b) configuration, we study linear and nonlinear (KAM) stability of $\gamma_{b}$. We have the following theorem.

\begin{theorem}
For every fixed $n \geq 3$, there exists an open interval in $\epsilon$ such that the $\gamma_{b}$ orbit in the billiard table $Q(R_{b})_{n, \epsilon}$ is KAM stable, with $R$ depending on $\epsilon$. Therefore each  billiard table in the sequence $Q(R_{b})_{n, \epsilon}$ is not ergodic, with $R_{b}$ decreasing to zero as $n \to \infty$.
\label{theorem2}
\end{theorem}

The proof of Theorem~\ref{theorem2} is given in Section~\ref{sec:level4}. While some heuristic and numerical papers on billiards treat the existence of linearly stable (elliptic) periodic orbits as a sufficient criterion to deduce the existence of elliptic islands (a set of invariant curves of positive measure surrounding the elliptic orbit) and hence non-ergodicity of the billiard, for a rigorous mathematical investigation of stability of elliptic orbits  one needs a more delicate analysis. Indeed, `linear ellipticity' is a fragile dynamical property: for example, it is known that in certain two dimensional maps, elliptic fixed points are not surrounded by invariant curves after a small perturbation \cite{katok1970new}. Thus one needs to consider the effect of higher order terms to ensure (local) stability  of periodic orbits.

To prove Theorem~\ref{theorem2}, we apply Birkhoff Normal Form with Moser's Twist Theorem \cite{siegel2012lectures}, which is a commonly used approach to study KAM stability in area-preserving maps. This technique has been used for establishing  stability of some periodic orbits in certain billiard systems before. The papers by Kamphorst et al \cite{kamphorst2005first, carneiro2003elliptic} established the stability of 2-periodic orbits in billiards with strictly convex $C^{5}$ boundaries, while Donnay \cite{donnay1996elliptic} showed the existence of elliptic islands in generalised Sinai billiards.  Rom-Kedar and Turaev \cite{turaev1998elliptic} proved the existence of islands for certain near-ergodic Hamiltonian flows limiting to a billiard flow and also for billiards with steep repelling potentials  \cite{rom1999big}. However, explicit computations with Birkhoff normal form are not feasible for an arbitrary billiard boundary, since a priori one needs to know its details (the form of the billiard map, and the location of the periodic orbit). Because we are dealing with circular boundaries, our task is tractable in this regard.

We show that the Birkhoff coefficient \cite{moser2001stable} of $\gamma_{b}$ periodic orbits is nonzero, which implies KAM stability, hence showing non ergodicity of the billiard dynamics. 

The paper is organised as follows. In Section~\ref{sec:level2} we review the basic theory of billiards necessary for the study of linear stability properties of our billiard tables. In Section~\ref{sec:level3}, we study the billiard geometry for type (a) periodic orbits and analytically prove Theorem~\ref{theorem1}. Section~\ref{sec:level4} is devoted to the study of type (b) orbits. First we show their linear stability by the same methods as in Section~\ref{sec:level3}. Then by using KAM theory and Birkhoff normal form, we prove Theorem~\ref{theorem2}. The appendices  provide the derivation of the billiard map required for computation of the Birkhoff coefficient. The appendices also include an auxiliary  Lemma~\ref{lem} used in the proof of Proposition~\ref{prop2}.

\section{\label{sec:level2}Preliminaries}

We state some standard facts from the theory of billiards and area-preserving maps. The following information may be found in Chernov \cite{chernov2006chaotic} or in Berry \cite{berry1981regularity}.

Let $Q \in \mathbb{R}^{2}$ be a bounded domain, with $C^{l}$-smooth, $l \geq 3$, boundary $\partial Q$. We call $Q$ the billiard table. An orientation of $\partial Q$ is such that $Q$ is to the left on $\partial Q$. The billiard phase space $M$ consists of the boundary $\partial Q$ and unit velocity vectors $\vec{v}$ pointing inwards of $\partial Q$.  A standard coordinate system on $M$ is $(s,\theta)$ where $s$ is the arc length parameter on $\partial Q$ and $\theta \in (0,\pi)$ is the angle between the positively oriented tangent to $\partial Q$ at  the point $s$ and the vector $\vec{v}$. Then $M$ is the Poincare section for the billiard flow, and we define billiard map $B: M \mapsto M$, $B(s,\theta)=(s_{1},\theta_{1})$. The billiard map preserves the measure $\sin \theta ds d\theta$ on $M$. and it is well known that $B$ is area-preserving in the coordinates $(s, \cos \theta)$. Define the signed curvature of $\partial Q$ by $\kappa = \kappa(s)$, such that $\kappa <0$ for convex boundaries, $\kappa>0$ for concave boundaries, and $\kappa=0$ for flat boundaries. Let $\tau$ denote the \textit{flight distance} between two consecutive collision points on the boundary, $s$ and $s_{1}$, $\kappa$ is the curvature at $s$ and $\kappa_{1}$ is the curvature at $s_{1}$. Then derivative of $B$ at $z=(s,\theta)$ is given by

\begin{equation}
DB(z) = -\left( \begin{array}{cc}
\ \frac{\tau \kappa + \sin \theta}{\sin\theta_{1}} & \frac{\tau}{\sin\theta_{1}} \\ 
\frac{\tau \kappa \kappa_{1}  + \kappa_{1} \sin \theta}{\sin\theta_{1}} + \kappa & \frac{\tau \kappa_{1}}{\sin\theta_{1}} + 1 \end{array} \right)
\label{billiardderivative}
\end{equation}
\
To study the linear stability of an $n$-periodic point $B^{n}(z_{0})=z_{0}$, where $z_{0}=(s_{0},\theta_{0})$, we need to examine the product of $n$-matrices of the above type

\begin{equation}
DB^{n}(z_{0})=DB(z_{n-1})DB(z_{n-2})...DB(z_{0})
\label{stability}
\end{equation}
\
The characteristic polynomial of $DB^{n}(z_{0})$ is of the form $\lambda^{2} - \lambda \mbox{tr}(DB^{n}(z_{0})) + 1,$ where $\{\lambda, \lambda^{-1}\}$ are eigenvalues of $DB^{n}(z_{0})$. The corresponding periodic point is said to be elliptic and linearly stable if $|\mbox{tr}(DB^{n}(z_{0}))|<2$, hyperbolic and unstable if $|\mbox{tr}(DB^{n}(z_{0}))|>2$ and parabolic (neutrally stable) if $|\mbox{tr}(DB^{n}(z_{0}))|=2$.

\section{\label{sec:level3}Stability analysis of type (a) orbits}

In this section we prove Theorem~\ref{theorem1}. Consider the billiard table $Q(R)$ as defined in the introduction with the orbit $\gamma_{a,k}$.
Define  $\delta \geq 0$ to be the parallel displacement of $D_{R}$ from the midpoint of the billiard trajectory segment and along it. $R$ and $\delta$ have to be chosen such that to ensure $D_{R}$ stays in the interior of $D$. Thus we obtain a $(R, \delta)$- parameter family of $(2n+2)$ periodic orbits for fixed $n \geq 3$ and $k$.

\begin{figure*} [t]
\centering
\includegraphics[width=12cm]{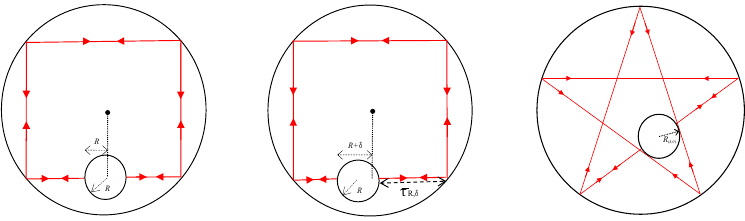}
\vspace*{-2mm}
\caption[]{Some periodic orbits $\gamma_{a,k}$; $k=1$, $n=4$, $\delta=0$ (left); $k=1$, $n=4$, $\delta \neq 0$ (middle); $k=2$, $n=5$ and $R_{k,0}$ as in eqn. (\ref{maxrn}) (right) }
\label{fig2}
\end{figure*}

The maximum value of $R$ depends on the choice of $\delta$, $k$ and $n$. From geometry, for a fixed $n \geq 3$, $\delta>0$ and $k=1$ we have the maximum possible $R=R_{\delta}$ such that $D_{R}$ avoids collision with the other parts of the same billiard trajectory:

\begin{equation}
R_{\delta} = 1- \sqrt{\delta^{2} + \cos^{2} \frac{\pi}{n}}
\label{maxdelta}
\end{equation}
\
For $\delta=0$, this expression yields $R_{0}$: 

\begin{equation}
R_{0} = 1 - \cos \frac{\pi}{n}
\label{rnm}
\end{equation}
\
which corresponds to $\partial D_{R_{0}}$ being \textit{tangent} to $\partial D$, thus forming a cusp.

When $k \neq 1$, with $(k,n)$ coprime, we have $n \geq 5$ and it is well known \cite{chernov2006chaotic} that the caustics for the orbit with the angle of reflection $\theta = \frac{k\pi}{n}$ in the disk $D$ (with $D_{R}$ removed) are just inner circles given  by the equation

$$x^{2}+y^{2} = \cos^{2} \frac{k \pi}{n}$$
\
Thus the billiard orbit produces a regular star $n$-gon  with  the inscribed tangent circle given by

$$x^{2}+y^{2} = \cos^{2} \frac{k \pi}{n}$$

It is simple to calculate that the length of the side of the $n$-gon is $2 \cos \frac{k\pi}{n} \tan \frac{\pi}{n}$.

For given $ n \geq 5$, $\delta \geq 0$ and $k>1$ we have the maximum possible radius $R = R_{k, \delta}$ for star orbits

\begin{equation}
R_{k, \delta} = \sin \frac{2\pi}{n}(\cos \frac{k\pi}{n}\tan \frac{\pi}{n} - \delta), \quad k>1
\label{maxradstar}
\end{equation}

This yields, for $\delta=0$

\begin{equation}
R_{k,0} = 2\cos \frac{k\pi}{n} \sin^{2} \frac{\pi}{n}, \quad k>1
\label{maxrn}
\end{equation}
\
We note that $R_{k,\delta}$ is such that the other segments of the billiard orbit do not hit $D_{R}$.

Define the map $B_{D}$ to be the composition of $(n-1)$ iterate of the well-known \cite{chernov2006chaotic} billiard map in a unit disk $D$:

\begin{equation}
B_{D}: (s, \theta) \mapsto (s + 2(n-1) \theta, \theta)
\label{bd}
\end{equation}
\
Define  $B_{in}$ to be the billiard map that takes the phase point $(s,\theta)$ with $s \in \partial D$ to  $(\bar{s}, \bar{\theta})$ where $ \bar{s} \in \partial D_{R}$, and define $B_{out}$ to be the map from $(\bar{s}, \bar{\theta})$ to a point $(\bar{\bar{s}}, \bar{\bar{\theta}})$ on $\partial D$ again. Thus, we may write the $(2n+2)$-periodic orbit as a square of the composition of  $B_{D}$, $B_{out}$ and $B_{in}$:

$$\left(B_{out} \circ B_{in} \circ B_{D}\right)^{2} = B^{2n+2}.$$

\begin{remark} Note that for linear stability computations, we do not require explicit formulae for $B_{in}$ and $B_{out}$ since we will be using the formula (\ref{billiardderivative}). However the explicit forms of $B_{in}$ and $B_{out}$  will be required for the study of nonlinear stability, and thus will be provided in appendix A. \end{remark}

The following Propositions~\ref{prop1} and ~\ref{prop2} constitute Theorem~\ref{theorem1}. 

\begin{proposition} Fix $k=1$, and $n \geq 3$. For $\delta=0$ and $R \leq R_{0}$, $\gamma_{a,1}$ is neutrally stable. For $\delta \neq0$ and $R < R_{\delta}$, the stability of $\gamma_{a,1}$ depends on the size of $\delta$ and $R$. In particular, for a small given $\delta>0$, $\gamma_{a,1}$ is linearly stable for $\frac{\sin \frac{\pi}{n} + \sqrt{\sin^{2} \frac{\pi}{n} + 4n^{2}\delta^{2}}}{2n} < R \leq R_{\delta}$. There is a saddle-center bifurcation at $R =\frac{\sin \frac{\pi}{n} + \sqrt{\sin^{2} \frac{\pi}{n} + 4n^{2}\delta^{2}}}{2n}$.  
\label{prop1}
\end{proposition}

\begin{proof}
Consider billiard geometry type (a), with $k=1$. Fix $n \geq 3$. We have a periodic orbit $\lbrace z_{0},...,z_{2n+1} \rbrace$, where $z_{i} = (s_{i}, \theta_{i})$. For $i \in \lbrace 0,...,n-1,n+1,...,2n \rbrace$, we have $s_{i} \in \partial D$, and for $i \in \ \lbrace n, 2n+1 \rbrace$, we have $s_{i} \in \partial D_{R}$. The initial condition is $z_{0} = (s_{0}, \frac{\pi}{n})$. We will calculate and establish the conditions on the trace of the derivative of the map $B^{2n+2}(z_{0})$ that ensure linear stability.

Assume that $D_{R}$ is displaced by $\delta \neq 0$ parallel to the orbit's segment in the direction of $z_{n-1}$. As is well known, \cite{chernov2006chaotic} the flight distance between two consecutive impact points $z_{i}$ and $z_{i+1}$ for $i \in \lbrace 0,1,..., (n-2), (n+1),..., (2n-1)\rbrace$  is $\tau =2\sin \frac{\pi}{n}$ since the collisions are on $D_{R}$. The flight distance between $z_{i}$ and $z_{i+1}$ for $i \in   \lbrace n-1,n\rbrace$ is $\tau_{R, \delta} =\sin \frac{\pi}{n} - R - \delta$, which corresponds to the length of the billiard trajectory segment between $\partial D$ and $\partial D_{R}$. Accordingly the flight distance for the reverse trajectory between $z_{i}$ and $z_{i+1}$, $i \in  \lbrace 2n, 2n+1 \rbrace$ is $\bar{\tau}_{R, \delta} = \sin \frac{\pi}{n} - R + \delta$. The signed curvature of $\partial D$ is $-1$, and the curvature of $\partial D_{R}$ is $\kappa=\frac{1}{R}$. 
\\
 
For $(n-1)$ consecutive bounces along the outer circle, we have the stability matrix 

\begin{equation}
DB_{D}(z_{0}) = \left( \begin{array}{cc}
\ 1 & -2(n-1) \\ 
0 & 1 \end{array} \right)
\label{outerbounce}
\end{equation}
\\

For the $n$-th bounce from $\partial D$ to $\partial D_{R}$, the stability matrix is

\begin{equation}
DB_{in}(z_{n-1}) = -\left( \begin{array}{cc}
\ -\tau_{R, \delta}  + \sin \frac{\pi}{n} & \tau_{R, \delta} \\ 
-\tau_{R, \delta} \kappa -1+ \kappa \sin \frac{\pi}{n}  & \tau_{R, \delta} \kappa + 1 \end{array} \right)
\end{equation}
\\
The stability matrix of the billiard map back from $\partial D_{R}$ to $\partial D$  is

\begin{equation}
DB_{out}(z_{n}) = -\left( \begin{array}{cc}
\ \frac{\tau_{R, \delta} \kappa + 1}{\sin \frac{\pi}{n}} & \frac{\tau_{R, \delta}}{\sin \frac{\pi}{n}} \\ 
\frac{-\tau_{R, \delta} \kappa -1}{\sin \frac{\pi}{n}}+ \kappa & \frac{-\tau_{R, \delta}}{\sin \frac{\pi}{n}} + 1 \end{array} \right)
\end{equation}
\\
Similar formulae follow for $DB_{D}(z_{n+1})$, $DB_{in}(z_{2n-1})$ and $DB_{out}(z_{2n})$. Using the expressions (\ref{billiardderivative}) and (\ref{stability}), we need to compute $\mbox{tr} \left(DB^{2n+2}(z_{0})\right)$, which turns out to be

\begin{equation}
\mbox{tr} \left(DB^{2n+2}(z_{0})\right)= 2 -\frac{16n\delta^{2} \left(nR^{2} - R\sin \frac{\pi}{n} - n\delta^{2}\right)}{R^{2}\sin^{2}\frac{\pi}{n}}
\label{tracetww}
\end{equation}
\\
Setting $\delta = 0$ shows parabolic stability of the corresponding orbit for all $R \leq R_{0}$. Note when $\partial D_{R}$ is tangent to $\partial D$, $\delta=0$ is necessary, since otherwise $D_{R}$ is no longer in the interior of $D$. This completes the proof of the first part of the proposition.

Let us consider the case $\delta \neq0$. Fix $n$ and small enough $\delta>0$. For linear stability, we need to ensure $|\mbox{tr} \left(DB^{2n+2}(z_{0})\right)| < 2$. From (\ref{tracetww}) a \textit{necessary} condition for the possibility of linearly stable orbits is 

\begin{equation}
nR^{2} - R \sin \frac{\pi}{n} - n\delta^{2}>0
\label{sizer}
\end{equation} 
\
This yields, upon rejecting the unphysical negative value, $R > \frac{\sin \frac{\pi}{n} + \sqrt{\sin^{2} \frac{\pi}{n} + 4n^{2}\delta^{2}}}{2n}$. Thus $R$ is determined from (\ref{maxdelta}) and (\ref{sizer}) by the inequalities

\begin{equation}
\frac{\sin \frac{\pi}{n} + \sqrt{\sin^{2} \frac{\pi}{n} + 4n^{2}\delta^{2}}}{2n} < R \leq R_{\delta}
\label{range}
\end{equation} 
\
Let us also show that (\ref{sizer}) is sufficient. Indeed, for sufficiency, (\ref{tracetww}) implies we need $-2 < 2 - \frac{16n\delta^{2} \left(nR^{2} - R\sin \frac{\pi}{n} - n\delta^{2}\right)}{R^{2}\sin^{2}\frac{\pi}{n}}$, which upon rearranging yields

\begin{equation}
\frac{R^{2}(\sin^{2} \frac{\pi}{n} - 4n^{2}\delta^{2})}{4n\delta^{2}} + R\sin \frac{\pi}{n} + n\delta^{2} >0
\label{negtracce}
\end{equation}
\
Denote by $f = \frac{\sin^{2} \frac{\pi}{n} - 4n^{2}\delta^{2}}{4n\delta^{2}}$ the coefficient of $R^{2}$ in (\ref{negtracce}). The \textit{formal} solutions of (\ref{negtracce}) are $R \in (-\infty, R^{-}) \cup (R^{+}, \infty)$ if $f > 0$ and $R \in (R^{+}, R^{-})$ if $f <0$, with $R > \frac{-n\delta^{2}}{\sin \frac{\pi}{n}}$ if $f = 0$. Here $R^{\pm} = \frac{-2\delta^{2}n}{\sin \frac{\pi}{n} \pm 2\delta n}$. In addition, we require that $R$ satisfies (\ref{range}). Consider the case $f>0$ corresponding to $\delta < \frac{\sin \frac{\pi}{n}}{2n}$. Then it is clear that $R^{\pm} < 0$. Since in the billiard context we have $R>0$, the \textit{physically allowed} solutions of (\ref{negtracce}) correspond to $R$ satisfying (\ref{range}). Consider now $f<0$, corresponding to $\delta > \frac{\sin \frac{\pi}{n}}{2n}$. Then $R^{+} < 0 < R^{-}$, so physically the possible domain for $R$ is $0 < R < R^{-}$. Tedious computations that we suppress show that $R^{-} > R_{\delta}$, hence the admissible values for $R$ lie in the set defined by (\ref{range}) indeed. The last case $f = 0$ that implies $\delta = \frac{\sin \frac{\pi}{n}}{2n}$ also leads to (\ref{range}). 

Hence $\gamma_{a,1}$ is linearly stable. Setting $R = \frac{\sin \frac{\pi}{n} + \sqrt{\sin^{2} \frac{\pi}{n} + 4n^{2}\delta^{2}}}{2n}$, we see that $\mbox{tr} \left(DB^{2n+2}(z_{0})\right)= 2$, so at this value of $R$ there is a saddle-center bifurcation, where the stability of $\gamma_{a,1}$ changes from hyperbolic to elliptic. This completes the proof of the proposition.
\end{proof}

\begin{figure}[h]
    \centering
    \includegraphics[width=0.4\textwidth]{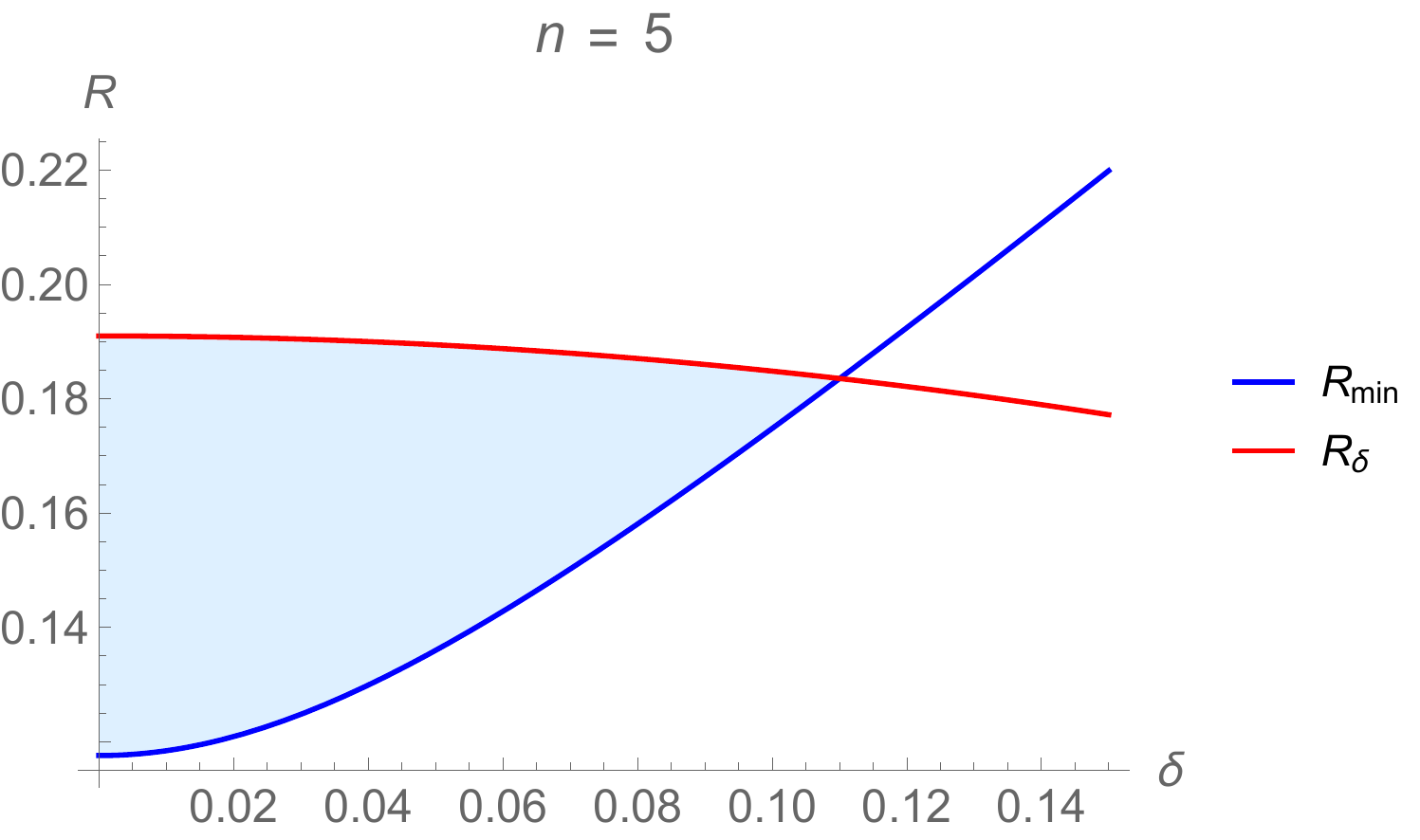}
    \includegraphics[width=0.4\textwidth]{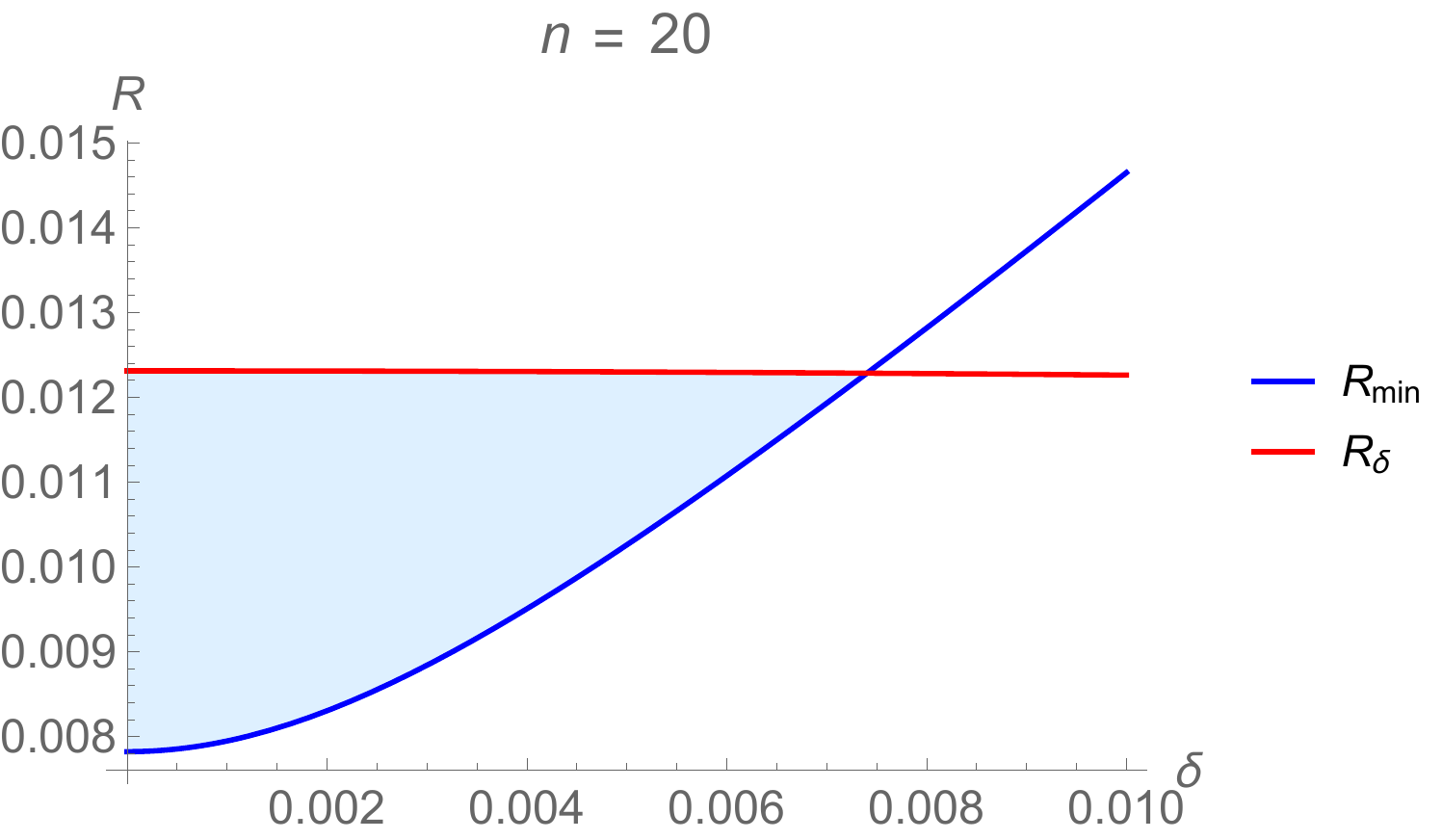}
    \caption{Admissible region for  linear stability of $\gamma_{a,1}$. The shaded region shows the range of $R$ defined by (\ref{range}) with $n=5,20$ for which $\gamma_{a,1}$ is linearly stable. Here $R_{min}=\frac{\sin \frac{\pi}{n} + \sqrt{\sin^{2} \frac{\pi}{n} + 4n^{2}\delta^{2}}}{2n}$ is the bifurcation value and the stability of $\gamma_{a,1}$ is parabolic on this curve. Below $R_{min}$, $\gamma_{a}$ is hyperbolic. For $n=5$, $R_{min}=R_{\delta}$ at $\delta=0.11004$, and for $n=20$, $R_{min}=R_{\delta}$ at $\delta=0.00740$. }
    \label{fig:sidebyside}
\end{figure}

Now consider star polygonal orbits: this is when $k \neq 1$ and $(k,n)$ are coprime. We have the following 

\begin{proposition} Let $n \geq 5$, and $k \leq \frac{n}{2}$, $(k,n)$ coprime. For $\delta =0$ and $R < R_{k,0}$, $\gamma_{a,k}$ is neutrally stable. For $\delta \neq 0$, and $R < R_{k,0}$, the stability of $\gamma_{a,k}$ depends on the relative size of $k$, $R$ and $\delta$. In particular, $\gamma_{a,k}$ is hyperbolic for $k \geq 7$. For $k < 7$ and small $\delta \neq 0$, $\gamma_{a,k}$ is linearly stable for $n \geq n_{k}$. Specifically, $n_{2}=5$, $n_{3}=9$, $n_{4}=13$, $n_{5}=21$ and $n_{6}=53$. There is a saddle-center bifurcation at $R=\frac{\sin \frac{k\pi}{n} + \sqrt{\sin^{2} \frac{k\pi}{n} + 4n^{2}\delta^{2}}}{2n}$. 
\label{prop2}
\end{proposition}

\begin{proof}

We again need to examine $\mbox{tr} \left(DB^{2n+2}(z_{0})\right)$, where we now have $z_{0} = (s_{0}, \frac{k\pi}{n})$, and the values of $\tau$, $\kappa$ are modified appropriately to account for $k \neq 1$ orbit configuration. The computations are identical to above, so we suppress them and proceed to give the result

\begin{equation}
\mbox{tr} \left(DB^{2n+2}(z_{0})\right)= 2 -\frac{16n\delta^{2}\left(nR^{2} - R\sin \frac{k\pi}{n} - n\delta^{2}\right)}{R^{2}\sin^{2}\frac{k\pi}{n}}
\label{tracethree}
\end{equation}

Setting $\delta=0$ we again see that the corresponding periodic orbit is parabolic. For $\delta \neq 0$, the same analysis as in the paragraph after (\ref{negtracce}) shows that the condition

$$nR^{2} - R\sin \frac{k\pi}{n} - n\delta^{2} >0$$

is necessary and sufficient for existence of linearly stable orbits. This yields the inequality $\frac{\sin \frac{k\pi}{n} + \sqrt{\sin^{2} \frac{k\pi}{n} + 4n^{2}\delta^{2}}}{2n} < R$. From (\ref{maxradstar}) we have $ \delta < \cos \frac{k\pi}{n} \tan \frac{\pi}{n}$. The same arguments as the ones following (\ref{negtracce}) imply that for given $k$ and $n$ with $0 < \delta \leq \frac{\sin \frac{k\pi}{n}}{2n}$, the allowed radius range is

\begin{equation}
\frac{\sin \frac{k\pi}{n} + \sqrt{\sin^{2} \frac{k\pi}{n} + 4n^{2}\delta^{2}}}{2n} < R \leq R_{k, \delta}
\label{boundsstar}
\end{equation}
\
while further laborious computations which we omit for the case $ \frac{\sin \frac{k\pi}{n}}{2n} < \delta < \cos \frac{k\pi}{n} \tan \frac{\pi}{n}$ lead to $\frac{\sin \frac{k\pi}{n} + \sqrt{\sin^{2} \frac{k\pi}{n} + 4n^{2}\delta^{2}}}{2n} < R \leq \tilde{R}$ where $\tilde{R}$ is at most $R_{k, \delta}$ (depending on the relative sizes of $n$ and $k$). Setting $R = \frac{\sin \frac{k\pi}{n} + \sqrt{\sin^{2} \frac{k\pi}{n} + 4n^{2}\delta^{2}}}{2n}$, we obtain $\mbox{tr} \left(DB^{2n+2}(z_{0})\right)= 2$ and thus there is a saddle-center bifurcation at this value of $R$ for a given $\delta$.

Let us find the range of $k$ for given $n$ such that (\ref{boundsstar}) is satisfied. Since we may take $\delta>0$ arbitrarily small, let us examine the limit $\delta \rightarrow 0$ in (\ref{boundsstar}) to facilitate the computation of admissible range of values of $k$ for a given $n$. Setting $\delta=0$ in (\ref{boundsstar}) yields

$$\frac{\sin \frac{k\pi}{n}}{n} < R < R_{k,0}$$

From which we obtain the inequality

\begin{equation}
2n \sin^{2} \frac{\pi}{n} > \tan \frac{k\pi}{n}
\label{inequalitystar}
\end{equation}

One needs to choose $k$ and $n$ such that this inequality is satisfied to obtain stable periodic orbits. Let us first examine (\ref{inequalitystar}) for large $n$. Expanding  (\ref{inequalitystar}) in Taylor series for $ n\rightarrow \infty $ gives the condition $2\pi > k$, i.e. $k \leq 6$ since $k \in \mathbb{Z}^{+}$.

Let us now determine admissible $k$ values more rigorously by examining (\ref{inequalitystar}) for any $n \geq 5$. Using the function $f$ in the Lemma~\ref{lem}  of Appendix C, we put $k = \lfloor f(n) \rfloor$, and we obtain that (\ref{inequalitystar}) is only satisfied for $k \leq 6$ for any $n$. Numerically we find that: $ k =2, n \geq 5; k=3, n \geq 9; k=4, n \geq 13; k=5, n \geq 21; k=6, n \geq 53$. Now $R_{k,0}$ is a decreasing function of $k$ while $\sin \frac{k\pi}{n}$ is increasing function of $k$, so if the inequality $nR_{k^{*},0} - \sin \frac{k^{*}\pi}{n}>0$ holds for some $k^{*}$, then it holds for all $k \leq k^{*}$. \end{proof}

\begin{remark}
Note that setting $\delta = 0$ makes $\gamma_{a,1}$ parabolic for all $R < R_{0}$, and $\gamma_{a,k}$ becomes parabolic for all $R < R_{k,0}$. Geometrically, $\delta = 0$ corresponds to a completely symmetric orbit. When $\delta  \neq 0$, the symmetry is lost, and the orbit $\gamma_{a,1}$ is only parabolic for $R =\frac{\sin \frac{\pi}{n} + \sqrt{\sin^{2} \frac{\pi}{n} + 4n^{2}\delta^{2}}}{2n}$, while $\gamma_{a,k}$ is only parabolic for $R=\frac{\sin \frac{k\pi}{n} + \sqrt{\sin^{2} \frac{k\pi}{n} + 4n^{2}\delta^{2}}}{2n}$; these values of $R$ are precisely the saddle-center bifurcation values given in Propositions~\ref{prop1} and ~\ref{prop2} respectively.
\end{remark}

\section{\label{sec:level4}Stability analysis of type (b) orbits}

We have established for the $Q(R)$ table cusp case that $\gamma_{a,1}$ orbits are parabolic. Now for the cusp geometry, it is possible to construct a type (b) $(2n+2)$- periodic orbit that corresponds to the case when the initial reflection angle on $\partial D$ is \textit{not} $\pi$-rational: $\theta_{0} \neq \frac{\pi}{n}$ (see Fig.~\ref{fig1}). We denote these orbits $\gamma_{b}$.  Again, we create a closed orbit by positioning $D_{R}$ in the orbit's path perpendicularly, such that the $R$ value for the tangency condition now reads

\begin{equation}
R_{b} = 1 + \frac{\cos \theta_{0}}{\cos n\theta_{0}}
\label{radiusb}
\end{equation}
\
We note that in this case $R_{b}$ depends on $\epsilon$ and the billiard table is $Q(R_{b})_{n, \epsilon}$. In the following two subsections, we will prove linear and nonlinear stability of $\gamma_{b}$, thus proving Theorem~\ref{theorem2}.

\subsection{\label{sec:level5}Linear stability of $\gamma_{b}$}

First, we investigate linear stability of $\gamma_{b}$. In the general case for $\theta_{0} \neq \frac{\pi}{n}$, the expression for $DB^{2n+2}$ is complicated and so it is difficult to draw any conclusions for the stability of the periodic orbit. Instead, let us pick $\epsilon >0$ and investigate the limit $\theta_{0} =  \frac{\pi}{n} + \epsilon$ as $\epsilon \rightarrow 0$.

\begin{proposition} There exists $\epsilon^{*}(n)>0$ such that for all $\epsilon< \epsilon^{*}(n)$, $\gamma_{b}$ orbits are linearly stable for the initial reflection angle $ \theta_{0} = \frac{\pi}{n} + \epsilon.$
\label{typebprop}
\end{proposition}

\begin{proof}
We consider the trace of $\left(DB_{out}DB_{in} DB_{D}(z_{0})\right)^{2}$ as before, modifying the values of $\tau$, $\kappa$, $R$ and $\theta_{0}$ as appropriate. Expanding the trace in Taylor series in $\epsilon$ with the aid of Mathematica, we find

\begin{multline}
\mbox{tr} \left(DB^{2n+2}(z_{0}) \right)= \\
2 - \frac{16n\epsilon \left(\cos\frac{\pi}{n} - n\cot\frac{\pi}{n} + n\cos\frac{\pi}{n} \cot\frac{\pi}{n}\right)}{\cos\frac{\pi}{n}-1}+O(\epsilon^{2})
\label{trb}
\end{multline}
\
Since $\left(\cos\frac{\pi}{n} - n\cot\frac{\pi}{n} + n\cos\frac{\pi}{n} \cot\frac{\pi}{n}\right) $ is negative, we may ensure that $\gamma_{b}$ is elliptic if we take a small enough positive $\epsilon$.
\label{proofb}
\end{proof}

\begin{remark}
The formula (\ref{trb}) implies for linear stability of $\gamma_{b}$, one  has to take $\epsilon \lesssim \frac{\cos \frac{\pi}{n} - 1}{4n(\cos\frac{\pi}{n} - n\cot\frac{\pi}{n} + n\cos\frac{\pi}{n} \cot\frac{\pi}{n})}$, which implies $\epsilon \simeq \frac{\pi^{2}}{4n^{3}(\pi - 2)}$ for very large $n$.
\label{remarkb}
\end{remark}

\subsection{\label{sec:level6}KAM stability of $\gamma_{b}$}

To show KAM stability of $\gamma_{b}$, we make use of the following well-known result regarding Birkhoff normal form:

\begin{widetext}
\begin{proposition} \cite{kamphorst2005first}, \cite{carneiro2003elliptic}, \cite{grigo2009billiards} 
Suppose that the map $B^{n}(s,p)$ is area-preserving and has an $n$-periodic point at $(0,0)$. Assume $B^{n}$ is $C^{k}$ with $k \geq 4$. Writing its Taylor expansion up to order $3$ in the neighbourhood of $(0,0)$,

\begin{equation}
 B^{n}(s,p) = \left( \begin{array}{ll}
         a_{10}s + a_{01}p+a_{20}s^{2} +a_{11}sp +...+a_{03}p^{3}\\
        b_{10}s + b_{01}p + b_{20}s^{2} + b_{11}sp + ...+b_{03}p^{3}\end{array} \right) + O(|(s,p)|^{4}) 
\label{Taylorf}        
\end{equation}

If the point $(0,0)$ is elliptic with eigenvalues $\lambda = \exp(\pm i \mu)$ satisfying the nonresonant condition $\lambda^{3}, \lambda^{4} \neq 1$, there is a real-analytic coordinate change that takes the map to its Birkhoff normal form $z \rightarrow \lambda z e^{iA|z|^{2}} + O(|z|^{4})$. The first Birkhoff coefficient $A$ is

\begin{equation}
A =  \Im{c}_{21}+ \frac{\sin{\mu}}{\cos{\mu}-1}\left(3|c_{20}|^{2} + \frac{2 \cos \mu -1 }{2\cos\mu +1} |c_{02}|^2\right)
\label{birk}
\end{equation}
\\

Where

$$ \Im{c}_{21} = \frac{1}{8}a_{10}[-a_{12} + 3\frac{b_{10}a_{03}}{a_{01}} - 3\frac{a_{01}b_{30}}{b_{10}} + b_{12}] - \frac{1}{8}b_{10}[a_{12} - 3\frac{a_{01}a_{30}}{b_{10}} - \frac{a_{01}b_{21}}{b_{10}} + 3b_{03}]$$

$$|c_{20}|^2 = \frac{1}{16}\sqrt{\frac{-a_{01}}{b_{10}}}[\frac{b_{10}}{a_{01}}a_{02} + a_{20} + b_{11}]^{2} + \frac{1}{16}\sqrt{\frac{-b_{10}}{a_{01}}}[\frac{a_{01}}{b_{10}}b_{20} + b_{02} + a_{11}]^{2}$$

$$|c_{02}|^{2} =\frac{1}{16}\sqrt{\frac{-a_{01}}{b_{10}}}[\frac{b_{10}}{a_{01}}a_{02} + a_{20} - b_{11}]^{2} + \frac{1}{16}\sqrt{\frac{-b_{10}}{a_{01}}}[\frac{a_{01}}{b_{10}}b_{20} + b_{02} - a_{11}]^{2}$$
\\

If $A$ is non-zero, the the fixed point is nonlinearly stable \cite{moser2001stable}. 
\end{proposition}
\end{widetext}

Let us compute $A$ for $\gamma_{b}$. The boundaries $\partial D$ and $\partial D_{R}$ are analytic except at the \textit{cusp} - which corresponds to the point tangency of $\partial D_{R}$ to $\partial D$. However, our periodic orbits are bounded away from the cusp, so we do not have to deal with this issue. Observe that the billiard geometry is symmetrical about the x-axis and $\gamma_{b}$ is also symmetric with respect to $D_{R}$. Define $z_{0} = (\pi + \pi/n + \epsilon(1-n), \pi/n + \epsilon)$ to be the fixed point of $B^{2n+2}$ corresponding to $\gamma_{b}$. We have $B_{out} \circ B_{in} \circ B_{D}(z_{0}) = B^{n+1}(z_{0})$ and $\quad B^{2n+2}(z_{0}) = (B^{n+1}(z_{0}))^{2}$, with the explicit expressions for $B_{out}$, $B_{in}$ and $B_{D}$ given in the appendix A. 

\begin{figure} [h]
\centering
\includegraphics[width=5 cm]{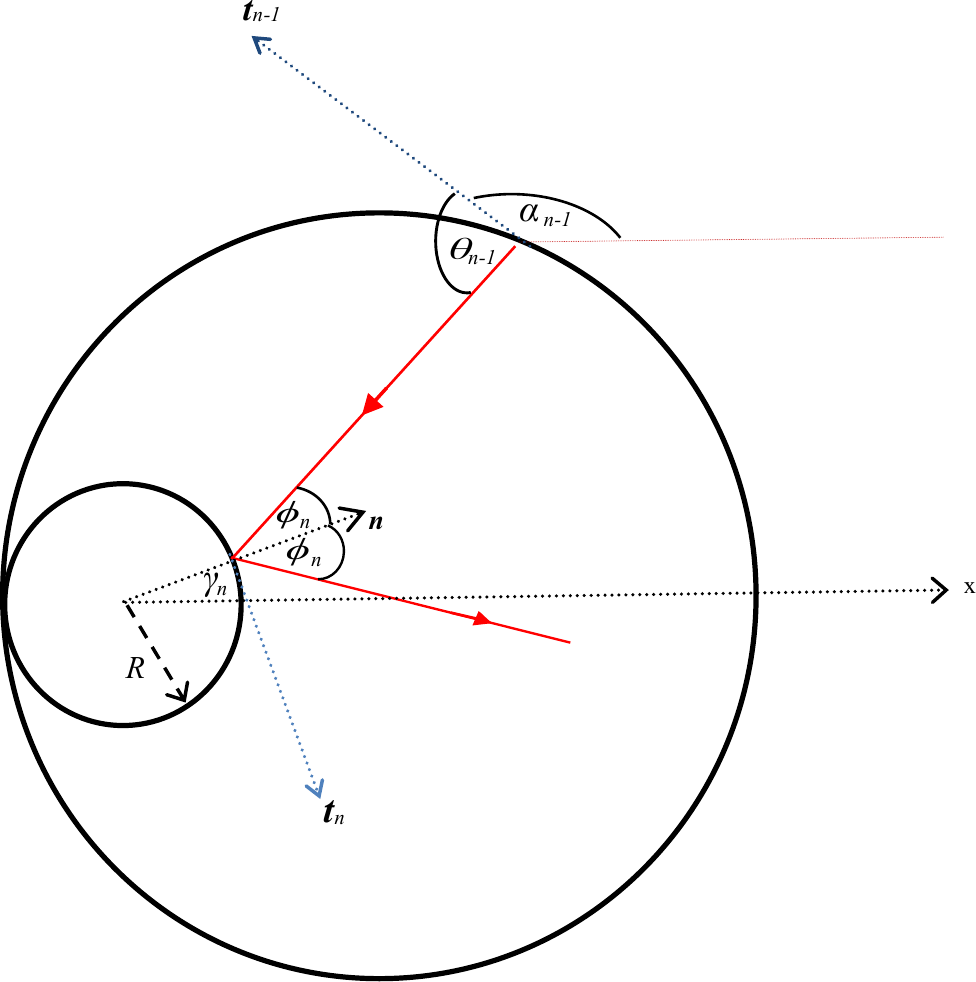}
\vspace*{-1mm}
\caption[]{Sketch of the billiard geometry for derivation of ~\eqref{eqn:bin} and ~\eqref{eqn:bout}. $D_{R}$ is symmetric on the x-axis. }
\label{fig4}
\end{figure}

We have the following

\begin{proposition}
For every fixed $n \geq 3$ and sufficiently small $\epsilon>0$, the point $z_{0} = ( \pi + \pi/n + \epsilon(1-n), \pi/n + \epsilon)$ is KAM stable for $B^{2n+2}$. 
\label{prop3}
\end{proposition}

\begin{proof}
We see that $z_{0}$ can be moved to the origin via  linear change of coordinates $Z = z-z_{0}$.  Let us define the \textit{Reflection map} $\mathcal{R}$

\begin{equation}
\mathcal{R}(s,\theta)=(-s,\pi-\theta)
\label{involution}
\end{equation} 
\

It is obvious that $\mathcal{R}$ is a diffeomorphism and $\mathcal{R}^{2}=\mbox{Id}$. Furthermore, it is a reversing involution for $B^{n+1}(z)$, since $\mathcal{R} \circ B=B\circ \mathcal{R}$. Observe that by composing $B_{out} \circ B_{in} \circ B_{D}(z_{0})$ with $\mathcal{R}(s,\theta)$, we obtain $z_{0}$. Thus $z_{0}$ is a fixed point of the map $\mathcal{R} \circ B^{n+1}$. Hence we have $B^{2n+2} = (\mathcal{R} \circ B^{n+1})^{2}$, and the stability of the fixed point of $\mathcal{R} \circ B^{n+1}$ corresponds to stability of the fixed point of $B^{2n+2}$. 

Let us check the properties of the linearized map $\mathcal{R} \circ B^{n+1}$. In particular, we are interested in the linear stability of $z_{0}$. Let us change coordinates $(s, \theta) \mapsto (s,r)\equiv (s, \cos \theta)$.  We remark that $\mathcal{R}$ in terms of $(s,r)$ is $\mathcal{R}(s,r)=(-s,-r)$, and thus indeed $\mathcal{R} \circ B^{n+1}$ is area-preserving. The tangent map is

\begin{equation}
D(\mathcal{R} \circ B^{n+1}) = \left( \begin{array}{cc}
\ \frac{\partial s}{\partial s_{0}}  & \frac{\partial s}{\partial r_{0}}\\ 
\frac{\partial r}{\partial s_{0}} & \frac{\partial r}{\partial r_{0}} \end{array} \right)
\label{linearized}
\end{equation}
\
The determinant of which is equal to $1$. For small enough $\epsilon$, the modulus of the trace of (\ref{linearized}) evaluated at $z_{0}$ is $ \left|-2 -2\epsilon \left(\frac{n \cos \frac{\pi}{n}}{\sin \frac{\pi}{n} \sin^{2} \frac{\pi}{n}} \left( n(\cos \frac{\pi}{n} -1) + \sin \frac{\pi}{n} \right) \right) + O(\epsilon^{2}) \right| < 2$ and hence $z_{0}$ is linearly stable. Let the eigenvalues of (\ref{linearized}) be $\lambda_{\pm}=\exp(\pm i \mu)= u \pm \mbox{i}v$. Then

$$\lambda_{\pm}= \frac{\frac{\partial s}{\partial s_{0}} + \frac{\partial r}{\partial r_{0}} \pm i\sqrt{4-\left(\frac{\partial s}{\partial s_{0}} + \frac{\partial r}{\partial r_{0}}\right)^{2}}}{2}$$
\
Which gives, since $\mu = \arctan \frac{v}{u}$,

\begin{eqnarray}
\mu = && -\frac{\sqrt{2 \epsilon}}{\sin \frac{\pi}{n}}\sqrt{n \left( n \sin \frac{2\pi}{n} - (1+2 \cos \frac{\pi}{n} + \cos \frac{2\pi}{n})\right)}\nonumber\\
&& + O(\epsilon^{3/2})
\end{eqnarray}
\
and $\mu$ tends to $0$ as $\epsilon$ tends to $0$, thus signifying in the limit $\epsilon=0$ we have the parabolic stability corresponding to $\gamma_{a}$ orbit, as expected. 

Using the explicit form of $B_{in}$, $B_{out}$ and $B_{D}$ given in the appendix A, eqns. ~\eqref{eqn:bdd},~\eqref{eqn:bin}, ~\eqref{eqn:bout} and appendix B for partial derivatives, we are able to find higher terms in Taylor expansion $R \circ B^{n+1}$. By plugging the resulting expressions in Mathematica and expanding for small $\epsilon>0$ and fixed $n \geq 3 $, we find the Birkhoff coefficient $A$ (\ref{birk}), is, to leading order:

\begin{widetext}
\begin{equation}
A(n,\epsilon) = \frac{5 \cosec^{5} \left(\frac{\pi}{2n}\right) \sec^{3} \left(\frac{\pi}{2n}\right)\sin^{5/2} \left(\frac{\pi}{n}\right) \left(\cos \left(\frac{\pi}{2n}\right) - n\sin \left(\frac{\pi}{2n}\right) \right)^{2}}{192 n^{2} \epsilon^{2} \sqrt{n\sin \left(\frac{2\pi}{n}\right) -\left(1 + 2 \cos \left(\frac{\pi}{n}\right) + \cos \left(\frac{2\pi}{n}\right) \right)   }} \sqrt{\frac{\sin \left(\frac{2\pi}{n}\right)}{ n\sin \left(\frac{\pi}{n}\right)-(1+ \cos \left(\frac{\pi}{n}\right)) }} + O(\epsilon^{-1})
\label{Asize}
\end{equation}
\end{widetext}
\

It is obvious that $A(1,\epsilon)$ and $A(2,\epsilon)$ is undefined. Numerically we see that $A(n,\epsilon) \neq 0$ for all $n \geq 3$. Hence $z_{0}$ is KAM stable.
\end{proof}

\begin{proof}[Proof of Theorem~\ref{theorem2}]
The KAM stability of $\gamma_{b}$ immediately follows from Proposition~\ref{prop3}. The existence of an open interval in $\epsilon$ for which the table $Q(R_{b})_{n, \epsilon}$ has an elliptic island follows from the form of Birkhoff coefficient (\ref{Asize}), as it is clear that by changing $\epsilon$ slightly, $A(n, \epsilon)$ stays nonzero. It is obvious that with increasing $n$, $R_{b}$ tends to $0$.
\end{proof}

Let us examine the coefficient in the limit $n \rightarrow \infty$. Expanding (\ref{Asize}) in Taylor series for $n \rightarrow \infty$ gives

\begin{equation}
A = \frac{5}{24 \epsilon^{2}}\left( \frac{\pi-2}{\pi^{2}} + \frac{\pi-1}{6n^{2}}\right) + O\left(\frac{1}{n^{4} \epsilon^{2}}\right)
\label{Alarge}
\end{equation}
\
i.e. a non-vanishing function of $n$ and $\epsilon$ that grows unboundedly as $\epsilon$ tends to $0$. 

Let us fix $\epsilon$ and plot $\tilde{A} = \lim_{\epsilon \to 0} \epsilon^{2}A$, ignoring terms of $O(\epsilon^{-1})$ as a function of $n$ for $3 \leq n \leq 1000$) with logarithmic scale for $n$ (see Fig.~\ref{fig5}). For $n=3$ we have $\tilde{A} = -\frac{5(-2+ \sqrt{3)}}{54(-1+ \sqrt{3})} =  \simeq 0.0338912$. As $n \to 1000$, $\tilde{A} \to 0.024 \simeq \frac{5(\pi-2)}{24 \pi^{2}}$, as expected from (\ref{Alarge}). These calculations imply that $A$ is, to leading order, $O(\epsilon^{-2})$ and thus tends to infinity as $\epsilon$  decreases and the period $n$ increases. 

\begin{figure} [t]
\centering
\includegraphics[width=8 cm]{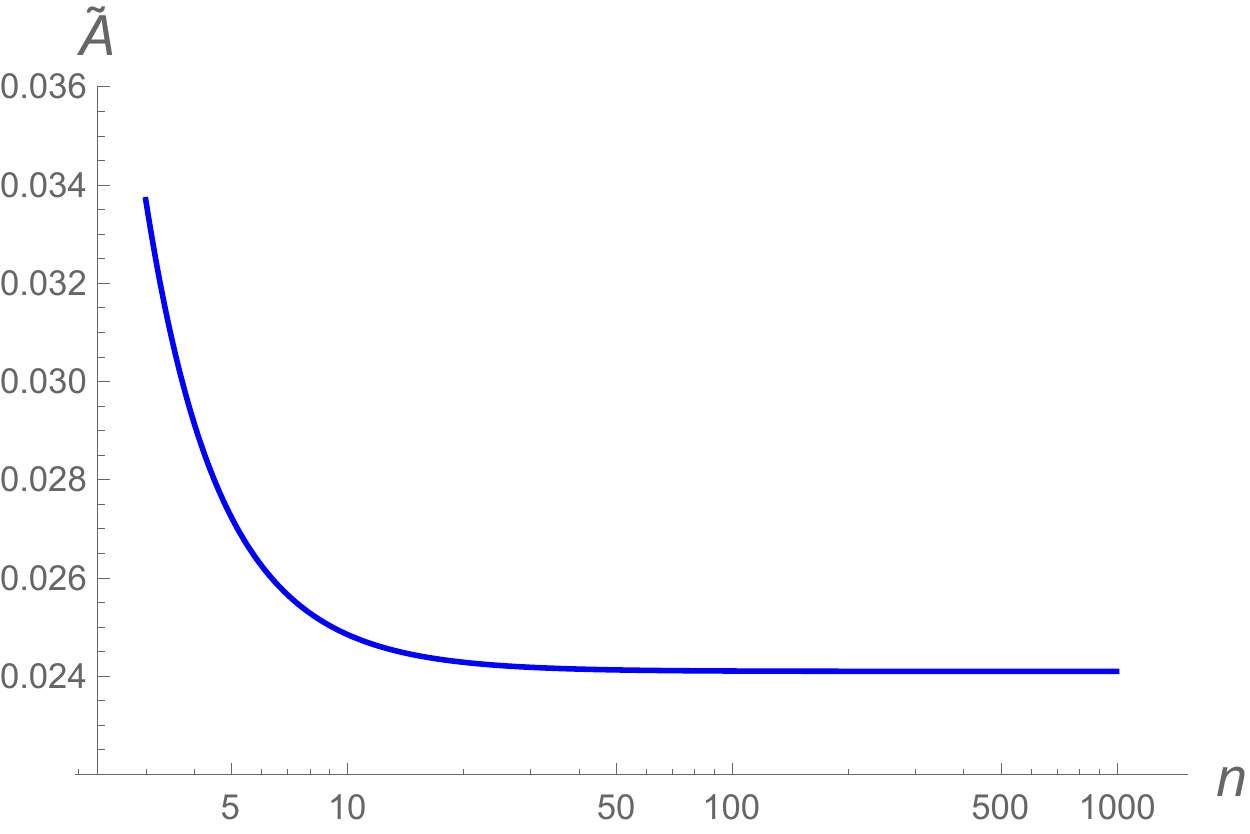}
\caption[]{A plot of Birkhoff coefficient for $\gamma_{b}$.  $\tilde{A} = \lim_{\epsilon \to 0} \epsilon^{2}A$ with $A$ as in eq. (\ref{Alarge})}
\label{fig5}
\end{figure}

\begin{remark} It is known from Grigo's thesis \cite{grigobilliards} that for certain small local perturbations of the scatterer boundary in the normal direction the elliptic periodic orbit will survive and remain nonlinearly stable. \end{remark}

\begin{remark}Constraining the radius $R = R_{b}$ given by (\ref{radiusb}) enabled us to make explicit computations with Birkhoff normal form in terms of $n$ and $\epsilon$ only. Setting $R < R_{b}$ while $n$ and $\epsilon$ are fixed will correspond to a nontangential position of $D_{R}$ to $D$. Similar computations to the ones in Section~\ref{sec:level5} show that the corresponding orbit will be linearly stable for $ \frac{sin \frac{\pi}{n}}{n} \lesssim R < R_{b}$. We expect that Proposition~\ref{prop3} will also hold for such $R$, however the Birkhoff coefficient will depend on $R$ and as such, relevant computations would be more laborious.  Therefore, we believe that there exists a sequence of non ergodic billiard tables with the scatterers of radius $\frac{sin \frac{\pi}{n}}{n} \lesssim R \leq R_{b}$ for every $n$.  \end{remark}

\begin{remark} We believe that analogous computations could be used to verify KAM stability of the $\gamma_{a,k}$ orbits for $\delta \neq 0$ corresponding to billiard tables where the scatterer is not tangent to the boundary. However we expect that the computations would be much more lengthy since $\delta \neq 0$ implies a loss of symmetry for the billiard orbit and as such the reduced billiard map $\mathcal{R} \circ B^{n+1}$ may not be utilised. \end{remark}

\section{\label{sec:level7}Summary and Conclusions}

We have studied the stability properties of some periodic orbits in a certain case of annular billiard, where the radius of the scatterer is very small compared to the external boundary. We also have considered a special limit when the scatterer is just tangent to the outer boundary, forming a cusp. This situation has thus far received relatively little attention, with no published rigorous results concerning the billiard dynamics in the regions formed by such cusps. The advantage of circular boundaries is that they allow one to obtain explicit formulae for the billiard map and perform perform direct computations to study linear and nonlinear stability of periodic orbits. We have established that given any arbitrary $n \geq 3$, the resulting $(2n+2)$-periodic orbit may be made linearly stable for an appropriate choice of scatterer radius and small displacement. Further, we have shown that for the cusp geometry, orbits with $\pi$-rational reflection angles are neutrally stable. We have found via the application of KAM theory that for the cusp geometry orbits with non $\pi$-rational angles can be nonlinearly stable. We have also found a neutrally stable configuration of $\gamma_{a}$ for a specific value of $R$ for non-symmetric scatterer position for a given $\delta \neq 0$, that corresponds to a direct parabolic bifurcation.
 
We note that the circular boundaries significantly simplified our investigation, and the straight forward application of KAM theory is unlikely to be feasible for other convex billiards with small tangential scatterers. Lazutkin's theorem \cite{lazutkin1973existence} implies non-ergodicity of strictly convex billiards.  One might wonder whether this property of a strictly convex billiard other than a circle will be maintained when a small tangential scatterer is introduced. This general problem seems much more difficult due to the curvature of the boundary no longer being constant.

We hope that this work will serve as a motivation for future investigation into billiards with cusps formed by a dispersing and a focusing arc.  There is a brief numerical investigation in \cite{da2015dynamics} into the scaling of the number of collisions for excursions into such a cusp, but as yet no published rigorous results.

\begin{acknowledgments}
The authors would like to thank Peter Balint, George Fullman, Kyle Guan and Dmitry Turaev for useful discussions. Carl Dettmann's research is supported by EPSRC grant EP/N002458/1. Vitaly Fain's research is supported by University of Bristol Science Faculty Studentship grant.
\end{acknowledgments}

\appendix

\section{Derivation of $B_{in}$ and $B_{out}$}

We give some details leading up to the expressions for $B_{out}$, $B_{in}$ and $B_{D}$ that were used in the computation of (\ref{birk}). We remark that similar formulae for the annular billiard map have been derived previously \cite{gouesbet2001periodic, saito1982numerical}. However we choose to derive the formulae used in our work since we are focusing on a very specific type of a billiard table with the scatterer tangential to the unit disk. The relevant sketch of the billiard geometry is in Fig.~\ref{fig4}.  We align $D_{R}$ and $D$ such that their centres fall on the horizontal axis $y=0$, and we position $D_{R}$ such that $\partial D_{R}$ is tangent to $\partial D$  since we study type (b) orbits. Let us parametrise $\partial D$ by $\varphi$ as 

$$\partial D = \partial D(\varphi) =  \{ (\cos \varphi, \sin \varphi): \quad \varphi \in (-\pi,\pi) \} $$ 
\\
with $\partial D$ traversed anticlockwise. Let us parametrise $\partial D_{R}$  by $\gamma$ where $\gamma$ is as in Fig. 4:

$$\partial D_{R} = \partial D_{R}(\gamma) = \{ (R \cos \gamma - (1-R), R\sin \gamma): \gamma \in (0,2\pi)  \}$$
Let us measure the arc length parameter on $\partial D_{R}$ clockwise, from the point of tangency of $\partial D_{R}$ to $\partial D$. The arc length $s$ in terms of $\gamma$ for $\partial D_{R}$  is thus

\begin{equation}
s = \pi + R(\pi - \gamma)
\label{arc2}
\end{equation}

Let us define $\phi = \pi/2 - \theta$, where $\phi \in [-\pi/2, \pi/2] $ is the angle made between the velocity vector of the particle at $\partial D_{R}$ and the normal $\mathbf{n}$,  chosen to be +ve in a clockwise direction, and $\theta \in [0,\pi]$ is the usual angle of reflection made with the +ve tangent vector $\mathbf{t}$ to $\partial D_{R}$. Also define $\alpha = \varphi + \pi/2$ as in Fig.~\ref{fig4}. 

Thus we have the billiard map as follows. For $(n-1)$ bounces along the circle, $B_{D}$ is obtained from (\ref{bd}), i.e.

\begin{align}
\label{eqn:bdd}
\begin{split}
 s_{n-1} &= s_{0} + 2(n-1)\theta_{0}
\\
\theta_{n-1} &= \theta_{0} .
\end{split}
\end{align}

Now from Fig.~\ref{fig4}. we have $\gamma_{n} = -\pi + \alpha_{n-1} + \theta_{n-1} - \phi_{n+1}$ and using this and (\ref{arc2}), we obtain $B_{in}: (s_{n-1}, \theta_{n-1}) \mapsto (s_{n}, \theta_{n})$:

\begin{align}
\label{eqn:bin}
\begin{split}
 s_{n} &= R(2\pi - (\theta_{n} + \theta_{n-1} + s_{n-1}) + \pi ,
\\
\theta_{n} &= \arccos \left( \frac{-\cos \theta_{n-1} - (1-R)\cos (\theta_{n-1} + s_{n-1})}{R}\right) .
\end{split}
\end{align}

Where subscript $n$ denotes the $n$-th impact which is on $\partial D_{R}$, and $(n-1)$-th impact is on $\partial D$, as above. Likewise  the map $B_{out}$ from scatterer to circle is obtained by reversing time, $B_{out}: (s_{n}, \theta_{n}) \mapsto (s_{n+1}, \theta_{n+1})$:

\begin{align}
\label{eqn:bout}
\begin{split}
 s_{n+1} &= \theta_{n} + \theta_{n+1} - \frac{s_{n}-\pi}{R}
\\
\theta_{n+1} &= \arccos \left( -R\cos \theta_{n} -(1-R)\cos \left(\theta_{n} - \frac{s_{n}-\pi}{R}\right)\right)
\end{split}
\end{align}

Now for $\gamma_{b}$ orbits (with $\theta_{0} = \pi/n + \epsilon$), we need the particle to collide with $\partial D_{R}$ perpendicularly, which can be achieved by choosing the initial position on the boundary $\partial D$ to be $s_{0} = -\pi + \pi/n + \epsilon (1-n)$ which corresponds to $\theta_{n-1} = \theta_{0} = \pi/n + \epsilon$, giving $z_{0}$. It can be checked (by implicit differentiation and calculating the Jacobian) that the above maps satisfy the symplecticity condition, if we convert to the coordinates $(s,r) \equiv (s,\cos \theta)$.

\section{Derivatives of $ \mathcal{R} \circ B^{n+1}$}

Let us compute, by the chain rule, the partial derivatives of $\mathcal{R} \circ B^{n+1}$ in coordinates $(s, r)$, that we use in Section~\ref{sec:level6} for computation of (\ref{birk}). We have $\mathcal{R} \circ B^{n+1}(s_{0}, r_{0}) = (s_{n+2}, r_{n+2})$.

To reduce typographical clutter, we write $s=s_{n+2}, \quad r=r_{n+2}, \quad \theta = \theta_{n+2}$. We do not require chain rule to compute $\frac{\partial s}{\partial s_{0}},\frac{\partial^{2} s}{\partial s^{2}_{0}}$ and $\frac{\partial^{3} s}{\partial s^{3}_{0}}$. The other derivatives are:

\begin{widetext}

\begin{equation}
\frac{\partial s}{\partial r_{0}} = -\frac{1}{\sin \theta_{0}} \frac{\partial s}{\partial \theta_{0}}; \qquad \frac{\partial^{2} s}{\partial s_{0} \partial r_{0}} = -\frac{1}{\sin \theta_{0}} \frac{\partial^{2} s}{\partial s_{0} \partial \theta_{0}}
\end{equation}
\
\begin{equation}
\frac{\partial^{2} s}{\partial r^{2}_{0}} = \frac{1}{\sin^{2} \theta_{0}} \left( \frac{\partial^{2} s}{\partial \theta^{2}_{0}} -\frac{\cos \theta_{0}}{\sin \theta_{0}} \frac{\partial s}{\partial \theta_{0}} \right); \qquad \frac{\partial^{3} s}{\partial s^{2}_{0} \partial r_{0}} = -\frac{1}{\sin \theta_{0}}\frac{\partial^{3} s}{\partial s^{2}_{0} \partial \theta_{0}}
\end{equation}
\
\begin{equation}
\frac{\partial^{3} s}{\partial r^{3}_{0}} = -\left(\frac{1}{\sin^{3} \theta_{0}} + \frac{3 \cos^{2} \theta_{0}}{\sin^{5} \theta_{0}}\right)\frac{\partial s}{\partial \theta_{0}} + \frac{3 \cos \theta_{0}}{\sin^{4} \theta_{0}}\frac{\partial^{2} s}{\partial \theta^{2}_{0}} - \frac{1}{\sin^{3}\theta_{0}}\frac{\partial^{3} s}{\partial \theta^{3}_{0}}; \qquad \frac{\partial^{3} s}{\partial r^{2}_{0} \partial s_{0}} = \frac{1}{\sin^{2} \theta_{0}} \left(\frac{\partial^{3} s}{\partial \theta^{2}_{0} \partial s_{0}} - \frac{\cos \theta_{0}}{\sin \theta_{0}} \frac{\partial^{2} s_{0}}{\partial s_{0} \partial \theta_{0}} \right)
\end{equation}
\
\begin{equation}
\frac{\partial r}{\partial s_{0}} = -\sin \theta_{0} \frac{\partial \theta}{\partial s_{0}}; \qquad \frac{\partial r}{\partial r_{0}} = \frac{\sin \theta}{\sin \theta_{0}}\frac{\partial \theta}{\partial \theta_{0}}
\end{equation}
\
\begin{equation}
\frac{\partial^{2} r}{\partial s^{2}_{0}} = -\cos \theta \left(\frac{\partial \theta}{\partial s_{0}}\right)^{2} - \sin \theta \frac{\partial^{2} \theta}{\partial s^{2}_{0}}; \qquad \frac{\partial^{2} r}{\partial r^{2}_{0}} = -\frac{\cos \theta}{\sin^{2} \theta_{0}}\left(\frac{\partial \theta}{\partial \theta_{0}}\right)^{2} + \frac{\cos \theta_{0} \sin \theta}{\sin^{3} \theta_{0}}\frac{\partial \theta}{\partial \theta_{0}} - \frac{\sin \theta}{\sin^{2} \theta_{0}}\frac{\partial^{2} \theta }{\partial \theta^{2}_{0}}
\end{equation}

\begin{equation}
\frac{\partial^{2} r}{\partial s_{0} \partial r_{0}} = \frac{\cos \theta}{\sin \theta_{0}}\frac{\partial \theta}{\partial s_{0}} \frac{\partial \theta}{\partial \theta_{0}} + \frac{\sin \theta}{\sin \theta_{0}}\frac{\partial^{2} \theta}{\partial s_{0} \partial \theta_{0}} \qquad \frac{\partial^{3} r}{\partial s^{3}_{0}} = \sin \theta \left(\frac{\partial \theta}{\partial s_{0}}\right)^{3} - 3\cos \theta \frac{\partial \theta}{\partial s_{0}}\frac{\partial^{2} \theta}{\partial s^{2}_{0}} - \sin \theta \frac{\partial^{3} \theta}{\partial s_{0}^{3}}
\end{equation}
\
\begin{equation}
\frac{\partial^{3} r}{\partial s^{2}_{0} \partial r_{0}} = \frac{\sin \theta}{\sin \theta_{0}}\left(\frac{\partial^{3} \theta}{\partial s^{2}_{0} \partial \theta_{0}} - \left(\frac{\partial \theta}{\partial s_{0}}\right)^{2}\frac{\partial \theta}{\partial \theta_{0}}\right) + \frac{\cos \theta}{\sin \theta_{0}}\left(\frac{\partial^{2} \theta}{\partial s^{2}_{0}}\frac{\partial \theta}{\partial \theta_{0}} + 2 \frac{\partial \theta}{\partial s_{0}}\frac{\partial^{2} \theta}{\partial s_{0} \partial \theta_{0}}\right)
\end{equation}

\begin{multline}
\frac{\partial^{3} r}{\partial r^{2}_{0} \partial s_{0}}=\frac{-\sin \theta}{\sin^{2} \theta_{0}}\left(\frac{\partial^{3} \theta}{\partial \theta^{2}_{0} \partial s_{0}} - \left(\frac{\partial \theta}{\partial \theta_{0}}\right)^{2} \frac{\partial \theta}{\partial s_{0}}\right) + \frac{\cos \theta_{0}}{\sin^{3} \theta_{0}}\left(\sin \theta \frac{\partial^{2} \theta}{\partial s_{0} \partial \theta_{0}} + \cos \theta \frac{\partial \theta}{\partial s_{0}} \frac{\partial \theta}{ \partial \theta_{0}}\right) \\
 - \frac{\cos \theta}{\sin^{2} \theta_{0}}\left(2 \frac{\partial^{2} \theta}{\partial s_{0} \partial \theta_{0}}\frac{\partial \theta}{\partial \theta_{0}} + \frac{\partial \theta}{\partial s_{0}} \frac{\partial^{2} \theta}{\partial \theta^{2}_{0}} \right) 
\end{multline}
\
\begin{multline}
\frac{\partial^{3} r}{\partial r^{3}_{0}} = \frac{\sin \theta}{\sin^{3} \theta_{0}}\left(-\left( \frac{\partial \theta}{\partial \theta_{0}}\right)^{3} + \frac{\partial^{3} \theta }{\partial \theta^{3}_{0}}\right) - \frac{3\cos \theta_{0}}{\sin^{4} \theta_{0}} \left(\cos \theta\left(\frac{\partial \theta}{\partial \theta_{0}}\right)^{2} + \sin \theta \frac{\partial^{2} \theta}{\partial \theta^{2}_{0}}\right) + \frac{3\cos \theta}{\sin ^{3} \theta_{0}}\frac{\partial^{2} \theta}{\partial \theta^{2}_{0}} \frac{\partial \theta}{\partial \theta_{0}} \\
+ \frac{\sin \theta}{\sin^{4} \theta_{0}}\frac{\partial \theta}{\partial \theta_{0}}\left(\sin \theta_{0} + \frac{3\cos^{2} \theta_{0}}{\sin \theta_{0}}\right)
\end{multline}

\end{widetext}

\section{Auxiliary Lemma for Proposition~\ref{prop1}}

\begin{lemma}
The function $f(x) = \frac{x}{\pi} \arctan \left(2x \sin^{2} \frac{\pi}{x} \right)$ is strictly increasing on $(1, \infty)$, and $\lim_{x \to \infty} f(x) = 2\pi$.
\label{lem}
\end{lemma}

\begin{proof}
Let us show that $f(x)$ is strictly increasing on $(1, \infty)$. Let us define the auxiliary function $$g(y) = (1+y^{2}) \arctan y - y.$$ Since $g(0)=0$ and $g'(y)= 2y \arctan y >0$ for $y>0$, $g(y)$ is strictly increasing and positive on $(0, \infty)$. Let $$y = 2x \sin^{2} \frac{\pi}{x}, \quad x>1.$$ Then rewriting $g$, we have

$$(1+4x^{2} \sin^{4} \frac{\pi}{x}) \arctan (2x \sin^{2} \frac{\pi}{x}) - 2x \sin^{2} \frac{\pi}{x} >0,$$
\
thus 

\begin{multline}
(1+4x^{2} \sin^{4} \frac{\pi}{x}) \arctan (2x \sin^{2} \frac{\pi}{x}) + 2x \sin^{2} \frac{\pi}{x} > 4x \sin^{2} \frac{\pi}{x} \\
> 4 \pi \sin \frac{\pi}{x} \cos \frac{\pi}{x}
\end{multline}
\
Now observe that

$$f'(x)= \frac{\arctan(2x \sin^{2} \frac{\pi}{x})}{\pi} + \frac{2x \sin^{2} \frac{\pi}{x} - 4 \pi \sin \frac{\pi}{x} \cos \frac{\pi}{x}}{\pi(1+4x^{2} \sin^{4} \frac{\pi}{x})} >0$$
\
by the above. Hence $f(x)$ is strictly increasing on $(1, \infty)$.

Since $\arctan x <x$ and $\sin x < x$ for $x>0$, we have that $f(x)$ is bounded above by $2\pi$ and $\lim_{x \to \infty} f(x) = 2\pi$.
\end{proof}

%

\end{document}